\documentclass[11pt, reqno]{amsart}
\usepackage{amsmath, amsthm, amscd, amsfonts, amssymb, graphicx, color}
\usepackage[bookmarksnumbered, colorlinks, plainpages]{hyperref}
\hypersetup{colorlinks=true,linkcolor=red, anchorcolor=green, citecolor=cyan, urlcolor=red, filecolor=magenta, pdftoolbar=true}

\newtheorem{theorem}{Theorem}[section]
\newtheorem{lemma}[theorem]{Lemma}
\newtheorem{prop}[theorem]{Proposition}
\newtheorem{cor}[theorem]{Corollary}
\theoremstyle{definition}

\theoremstyle{remark}
\newtheorem{remark}[theorem]{\bf{Remark}}
\numberwithin{equation}{section}
\allowdisplaybreaks
\begin{document}
	
\title[ A new norm on the space of reproducing kernel Hilbert space operators ]  
	 { A new norm on the space of reproducing kernel Hilbert space operators and Berezin number inequalities }

\author[R. K. Nayak and P. Bhunia ]{Raj Kumar Nayak and  Pintu Bhunia}
	
	\address{(Nayak) Department of Mathematics, GKCIET, Malda, India}
	\email{rajkumar@gkciet.ac.in / rajkumarju51@gmail.com}

    \address{(Bhunia) Department of Mathematics, Indian Institute of Science, Bengaluru 560012, Karnataka, India}
 \email {pintubhunia5206@gmail.com / pintubhunia@iisc.ac.in}

	\thanks{The second author was supported by National Post Doctoral Fellowship PDF/2022/000325 from SERB (Govt.\ of India) and SwarnaJayanti Fellowship SB/SJF/2019-20/14  (PI: Apoorva Khare) from SERB (Govt.\ of India). He also would like to thank National Board for Higher Mathematics (Govt.\ of India) for the financial support in the form of NBHM Post-Doctoral Fellowship 0204/16(3)/2024/R\&D-II/6747 under the mentorship of Prof. Apoorva Khare.}
	
	\subjclass[2020] {47A30; 15A60, 47A12}
	\keywords{Berezin symbol, Berezin number, reproducing kernel Hilbert space}

	\date{}
	
	\begin{abstract}
	In this note, we introduce a novel norm, termed the $t-$Berezin norm, on the algebra of all bounded linear operators defined on a reproducing kernel Hilbert space $\mathcal{H}$ as
    $$\|A\|_{t-ber} = \sup_{ \lambda, \mu \in \Omega} \left\{ t|\langle A \hat{k}_\lambda, \hat{k}_\mu \rangle| + (1-t) |\langle A^* \hat{k}_\lambda, \hat{k}_\mu \rangle| \right\}, \quad t\in [0,1],$$
    where $A \in \mathcal{B}(\mathcal{H})$ is a bounded linear operator. This norm characterizes those invertible operators which are also unitary.
	 Using this newly defined norm, we establish various upper bounds for the Berezin number, thereby refining the existing results. Additionally, we derive several sharp bounds for the Berezin number of an operator via the Orlicz function.
	\end{abstract}
	\maketitle

	\section{Introduction}
    
	\noindent The main purpose of this work is to introduce a new norm on a functional Hilbert space operator, which is always bounded by the classical Berezin norm. This new norm is highly effective in the analysis of various operators. Moreover, we establish that this newly defined norm yields enhanced bounds for the Berezin number compared to the existing well-known bounds. Note that here our study is inspired by the recent article by Bhunia \cite{Pintu GMJ}. Let us first introduce the necessary notation and terminologies.

Let $\Omega$ be a nonempty set. 
 A functional Hilbert space $\mathcal{H} = \mathcal{H}(\Omega)$ is a Hilbert space consisting of complex-valued functions on $\Omega,$ characterized by the property that the point evaluations are continuous. Specifically, for all $\lambda \in \Omega$ the map $f \rightarrow f(\lambda) $ is a continuous linear functional on $\mathcal{H}.$ The Riesz representation theorem states that for all $\lambda \in \Omega$ there is a unique element $k_\lambda \in \mathcal{H}$ such that $f(\lambda) = \langle f, k_{\lambda} \rangle,$ for all $f \in \mathcal{H}.$ The family $\{k_{\lambda}: \lambda \in \Omega  \}$ is referred to as the reproducing kernel  of $\mathcal{H}.$ If $\{e_n\}$ is an orthonormal basis for a functional Hilbert space $\mathcal{H},$ then the reproducing kernel of $\mathcal{H}$ is given by $k_{\lambda}(z)= \sum_n \bar{e_n}(\lambda) e_n(z);$ (see \cite[Problem 37]{Halmos}). For $\lambda \in \Omega,$  $\hat{k}_\lambda = \frac{k_\lambda}{\|k_\lambda\|}$ is the normalized reproducing kernel on $\mathcal{H}.$ 
	For a bounded linear operator $A \in \mathcal{B}(\mathcal{H}),$ the function $\tilde{A}$, defined on $\Omega$ by $\tilde{A}(\lambda) = \langle A \hat{k}_\lambda, \hat{k}_\lambda  \rangle$, is the Berezin symbol of $A.$ The Berezin set and the Berezin number of the operator $A$ are defined by 
		$\textbf{Ber}(A) = \{\langle A \hat{k}_\lambda, \hat{k}_\lambda \rangle: \lambda \in \Omega \} $
		and $\textbf{ber}(A) = \sup \{| \langle A \hat{k}_\lambda, \hat{k}_\lambda \rangle|: \lambda \in \Omega \}$,
	respectively. The Berezin norm of an operator $A \in \mathcal{B}(\mathcal{H})$ is defined by  $\|A\|_{ber} = \sup \left\{\left|\left\langle A \hat{k}_\lambda, \hat{k}_\mu \right\rangle\right|: \lambda, \mu \in \Omega\right\} .$  It is straightforward to verify that the Berezin norm $\|\cdot\|_{ber}$ defines a norm on $\mathcal{B}(\mathcal{H}).$ Additionally, it is evident that $\textbf{ber}(A) \leq \|A\|_{ber} \leq \|A\|. $ Moreover, the Berezin norm satisfies $\|A^*\|_{ber} = \|A\|_{ber}.$ The Berezin number of an operator $A \in \mathcal{B}(\mathcal{H})$ satisfies the following properties: 
	\begin{eqnarray*}
		~~&&(i)~~\textbf{ber}(\alpha A) = |\alpha| \textbf{ber}(A)~~\forall~~\alpha \in \mathbb{C};\\
		~~&&(ii)~~\textbf{ber}(A+B) \leq \textbf{ber}(A) + \textbf{ber}(B).
	\end{eqnarray*}
Note that $\textbf{ber}(\cdot) : \mathcal{B}(\mathcal{H})\to\mathbb{R}$ does not define a norm, in general. If $\mathcal{H}$ has the ``Ber" property (i.e. for any two operators $A,B \in \mathcal{B}\left(\mathcal{H}\right)$,   $\widetilde{A}\left(
\lambda\right)=\widetilde{B}\left(\lambda\right)$ for all $\lambda \in \Omega$ implies $A=B$), then $\textbf{ber}(\cdot)$ defines a norm on $\mathcal{B}(\mathcal{H}).$
Recently, Bhunia et al. proved in \cite[Proposition 2.11]{CAOT Pintu} that if $A\in \mathcal{B}(\mathcal{H})$ is positive, then 
    \begin{equation}\label{Pintu CAOT}
    \textbf{ber}(A) = \|A\|_{ber}.
   \end{equation}
   This equality does not hold for any self-adjoint operator, see e.g. \cite[Example 2.12]{CAOT Pintu}. 
The Berezin symbol (and Berezin number) is useful in studying reproducing kernel Hilbert space operators and it has wide application in operator theory. It has been studied in detail for Toeplitz and Hankel operators on Hardy and Bergman spaces. For an account of the Berezin symbol and its applications, we refer to \cite{K2013, KAR_JFA_2006,KS_CVTA_2005,T_OM_2021}. 
Recall that (see e.g. \cite{Paulsen_book}) the Hardy-Hilbert space of the unit disk  $\mathbb{D}= \{\lambda \in \mathbb{C} : |\lambda |<1 \}$ is denoted by ${H}^2(\mathbb{D}),$ and is defined as the Hilbert space of all analytic functions on $\mathbb{D}$ having power series representations with square summable complex coefficients. It is well known that ${H}^2(\mathbb{D})$ is a reproducing kernel Hilbert space, and for $\lambda \in \mathbb{D},$ the corresponding reproducing kernel of ${H}^2(\mathbb{D})$ is given by $k_{\lambda}(z)=\sum_{n=0}^{\infty}\bar{\lambda}^n z^n$.

\noindent We now introduce the $t-$Berezin norm on $\mathcal{B}(\mathcal{H}).$ For $A \in \mathcal{B}(\mathcal{H})$ and $t\in [0,1]$, the $t-$Berezin norm of $A$, denoted by $\|A\|_{t-ber}$, is defined as
$$\|A\|_{t-ber} = \sup_{ \lambda, \mu \in \Omega} \left\{ t|\langle A \hat{k}_\lambda, \hat{k}_\mu \rangle| + (1-t) |\langle A^* \hat{k}_\lambda, \hat{k}_\mu \rangle| \right\}.$$
Clearly, $ \textbf{ber}(A) \leq \|A\|_{t-ber} \leq \|A\|_{ber}$ for every $t\in [0,1].$ To illustrate, we consider an example. 
Suppose $M_z\in \mathcal{B}({H}^2(\mathbb{D}))$ with $M_z(f)=zf$ for all $f\in {H}^2(\mathbb{D}).$ Then we have
\begin{eqnarray*}
    1=\textbf{ber}(M_z) \leq  \|M_z\|_{t-ber} &=&\min_{t\in [0,1]} \sup_{\lambda,\mu \in \mathbb{D}} \sqrt{(1-|\lambda|^2) (1-|\mu|^2)} \frac{t|\mu|+(1-t)|\lambda|}{|1-\overline{\lambda} \mu|}\\
    &\leq& \sup_{\lambda,\mu \in \mathbb{D}} \sqrt{(1-|\lambda|^2) (1-|\mu|^2)} \frac{|\mu|}{|1-\overline{\lambda} \mu|}=  \|M_z\|_{ber}.
\end{eqnarray*}

In this work, we study various properties of the $t-$Berezin norm, and show that this norm characterizes those invertible operator which are unitary.
 We establish various upper bounds for the $t$-Berezin norm of operators as well as operator matrices, and conclude that this norm gives better upper bound of the Berezin number than the existing bounds. 
 Further, we present significant generalizations and improvements of the Berezin number inequalities for bounded linear operators utilizing the Orlicz functions.
 
	\section{The $t-$Berezin norm of operators}
    

 We begin our study with some basic properties of the $t-$ Berezin norm, which immediately follows from its definition.

	\begin{prop} \label{basic properties}
	If $A \in \mathcal{B}(\mathcal{H})$, then the following results hold:
	\begin{enumerate}
		\item $\| A\|_{t-ber} = \| A^*\|_{t-ber}.$
        \item  $\| A\|_{t- ber} =0 $ if and only if $A=0.$
		\item $\|\lambda A\|_{t- ber} =|\lambda| \| A\|_{t- ber} $ for all $\lambda \in \mathbb{C}.$
		\item $\| A\|_{t- ber} = \| A\|_{(1-t)- ber}. $
        \item $\| A+B\|_{t- ber} \leq \| A\|_{t- ber} + \| B\|_{t- ber}$ for every $B\in \mathcal{B}(\mathcal{H}).$
	\end{enumerate}
	\end{prop}
    
From this proposition, we can say that $\|\cdot\|_{t-ber}$ defines a norm on $\mathcal{B}(\mathcal{H})$ and is equivalent to the Berezin norm via the relation
\begin{equation} \label{relation inequality}
			\frac{1}{2} \| A\|_{ber} \leq \max \{t, 1-t\} \| A\|_{ ber} \leq  \| A\|_{t- ber} \leq \| A\|_{ber}.  
			\end{equation}


The above inequalities are sharp. Consider $A = \begin{pmatrix}
    0&1\\
    0&0
\end{pmatrix}$ is an operator on the reproducing kernel Hilbert space $\mathbb{C}^2,$ then $\|A\|_{ber} =1$  and $\|A\|_{t-ber} = \max\{t, 1-t\}$.
We now study an equivalent characterization for the equality of the last inequality in \eqref{relation inequality}.

	\begin{theorem}\label{th 5}
		If $A\in \mathcal{B}(\mathcal{H})$, then the following two statements are equivalent:\\
		(1) $\|A\|_{t-ber}=\|A\|_{ber}.$\\
		(2) There exist sequences $\{ \tau_n\}$ and $\{\eta_n\}$ in $\Omega$ such that \[\lim_{n \to \infty} |\langle A\hat{k}_{\tau_n},\hat{k}_{\eta_n}\rangle|=    \lim_{n \to \infty} | \langle A^*\hat{k}_{\tau_n},\hat{k}_{\eta_n}\rangle| =\|A\|_{ber}.\]	
	\end{theorem}
	\begin{proof}
	First we prove	$(1) \implies (2)$: Suppose $\|A\|_{t-ber}=\|A\|_{ber}$\  holds. As $A$ is bounded, there exist sequences of normalized kernels  $\{\hat{k}_{\tau_n}\}$ and $\{\hat{k}_{\eta_n}\}$ in $\mathcal{H}$ such that $$\lim_{n \to \infty} (t|\langle A\hat{k}_{\tau_n},\hat{k}_{\eta_n}\rangle|+ (1-t)| \langle A^*\hat{k}_{\tau_n},\hat{k}_{\eta_n}\rangle|) = \|A\|_{t-ber}.$$
		Therefore, we have 
		\begin{eqnarray*}
			\|A\|_{ber} &=& \|A\|_{t-ber} 
			= \lim_{n \to \infty} (t|\langle A\hat{k}_{\tau_n},\hat{k}_{\eta_n}\rangle|+ (1-t)| \langle A^*\hat{k}_{\tau_n},\hat{k}_{\eta_n}\rangle|\\
			&=& t \lim_{n \to \infty} |\langle A\hat{k}_{\tau_n},\hat{k}_{\eta_n}\rangle|+   (1-t) \lim_{n \to \infty} | \langle A^*\hat{k}_{\tau_n},\hat{k}_{\eta_n}\rangle|\\
			&\leq& t \|A\|_{ber} + (1-t) \|A\|_{ber}= \|A\|_{ber}.
		\end{eqnarray*}
		This implies \[\lim_{n \to \infty} |\langle A\hat{k}_{\tau_n}, \hat{k}_{\eta_n}\rangle|=    \lim_{n \to \infty} | \langle A^*\hat{k}_{\tau_n}, \hat{k}_{\eta_n}\rangle| =\|A\|_{ber}.\]
		
		We now prove \noindent $(2) \implies (1)$:
		Suppose there exist sequences $\{ \tau_n\}$ and $\{\eta_n\}$ in $\Omega$  such that 
 $\lim_{n \to \infty} |\langle A\hat{k}_{\tau_n},\hat{k}_{\eta_n}\rangle|=    \lim_{n \to \infty} | \langle A^*\hat{k}_{\tau_n}, \hat{k}_{\eta_n}\rangle| =\|A\|_{ber}.$
		Then, we have
		\[\|A\|_{t-ber} 
		\geq  \lim_{n \to \infty} (t|\langle A\hat{k}_{\tau_n}, \hat{k}_{\eta_n}\rangle|+ (1-t)| \langle A^*\hat{k}_{\tau_n}, \hat{k}_{\eta_n}\rangle| = \|A\|_{ber}.\]
		From this and using \eqref{relation inequality}, we conclude $\|A\|_{t-ber}=\|A\|_{ber}.$ 
	\end{proof}

It is observed that the $t-$Berezin norm is not an algebra norm, that is, for any two operators $A, B \in \mathcal{B}(\mathcal{H}),$ the inequality $\|AB\|_{t-ber} \leq \|A\|_{t-ber} \|B\|_{t-ber}$ does not necessarily hold. To illustrate this, consider $A =\begin{pmatrix}
	    0&1\\
        0&0
	\end{pmatrix}$ and  $ B = \begin{pmatrix}
	    0&0\\
        1&0
	\end{pmatrix}$ are operators on the two-dimensional reproducing kernel Hilbert space $\mathbb{C}^2.$ Then $ \|A\|_{t-ber} = \|B\|_{t-ber} =\max\{t, 1-t\}< 1,$ for $t \neq 0,1$,  whereas $\|AB\|_{t-ber} =1.$ In this connection, we now derive an inequality for $t-$Berezin norm  (involving the spectral radius) for the product of two operators.


 To establish this, we first note the following well-known lemma:
\begin{lemma}\label{mixed schwarz} \cite{Res}
		Let $A, B \in \mathcal{B}(\mathcal{H})$ with $|A|B = B^*|A|.$   Let $\psi, \eta$ be two nonnegative continuous functions on $[0, \infty) $ satisfying $\psi(t) \eta(t)=t.$ Then for any $x, y \in \mathcal{H}$ \[|\langle Ax, y \rangle| \leq r(B)\left\| \psi(|A|)x\right\| \left\|\eta(|A^*|)y \right\|.\]
        
        In particular, for $\psi(t)= \eta(t)=\sqrt{t}, \, t\geq 0$, we have  
        $|\langle Ax, y \rangle| \leq r(B)\left\| |A|^{1/2}x\right\| \left\||A^*|^{1/2}y \right\|.$
	\end{lemma}

\begin{theorem} \label{th 3}
		If $A, B \in \mathcal{B}(\mathcal{H})$ with $|A|B = B^*|A|$, then \[ \| AB\|_{t-ber} \leq r(B) \sqrt{\left\| t|A| + (1-t) |A^*| \right\|_{ber}^{} \left\| t|A^*| + (1-t) |A|\right\|_{ber}^{}}.\] 
        
        In particular, for $t=1/2,$
        \[ \| AB\|_{\frac12 -ber} \leq \frac12 r(B)\left\| |A| + |A^*| \right\|_{ber}^{}.\]
\end{theorem}
    
	\begin{proof}
	Following Lemma \ref{mixed schwarz}, for any two normalized reproducing kernels $\hat{k}_\lambda, \hat{k}_\mu$ in $\mathcal{H},$ we have that 
	$|\langle AB \hat{k}_\lambda, \hat{k}_\mu \rangle| \leq r(B) \left\||A|^{\frac{1}{2} }\hat{k}_\lambda \right\| \left\||A^*|^{\frac{1}{2} }\hat{k}_\mu \right\|. $
	Using this, we get
	\begin{eqnarray*}
		&&t |\langle AB \hat{k}_\lambda, \hat{k}_\mu \rangle|  + (1-t) |\langle  \hat{k}_\lambda, AB\hat{k}_\mu \rangle| \\
		&\leq& r(B) t \langle |A| \hat{k}_\lambda, \hat{k}_\lambda \rangle^{\frac{1}{2}} \langle |A^*| \hat{k}_\mu, \hat{k}_\mu \rangle^{\frac{1}{2}}+ r(B) (1-t) \langle |A^*| \hat{k}_\lambda, \hat{k}_\lambda \rangle^{\frac{1}{2}} \langle |A| \hat{k}_\mu, \hat{k}_\mu \rangle^{\frac{1}{2}}\\
	&\leq&	r(B) \left\{t  \langle |A| \hat{k}_\lambda, \hat{k}_\lambda \rangle + (1-t)  \langle |A^*| \hat{k}_\lambda, \hat{k}_\lambda \rangle\right\}^{1/2}  \left\{t\langle |A^*| \hat{k}_\mu, \hat{k}_\mu \rangle + (1-t)  \langle |A| \hat{k}_\mu, \hat{k}_\mu \rangle \right\}^{1/2}\\
	&& \,\,\,\,\,\,\,\,\,\,\,\,\,\,\,\,\,\,\,\,\,\,\,\,\,\,\,\,\,\,\,\,\,\,\,\,\,\,\,\,\,\,\,\,\,\,\,\,\,\,\,\,\,\,\,\,\,\,\,\,\,\,\,\,\,\,\,\,\,\,\,\,\,\,\,\,\,\,\,\,\,\,\,\,\,\,\,\,\,\,\,\,\,\,\,\,\,\,\,\,\,\,\,\,\,\,\,\,\,\,\,\,(\mbox{from the Cauchy-Schwarz inequality})\\
	&=& r(B)  \left\langle \left(t |A| + (1-t) |A^*|\right)\hat{k}_\lambda, \hat{k}_\lambda \right \rangle ^{\frac{1}{2}}  \left\langle \left(t |A^*| + (1-t) |A|\right)\hat{k}_\mu, \hat{k}_\mu \right \rangle^{\frac{1}{2}} \\
	&\leq&  r(B) \left\| t|A| + (1-t) |A^*| \right\|_{ber}^{{1}/{2}} \left\| t|A^*| + (1-t) |A|\right\|_{ber}^{{1}/{2}}.
	\end{eqnarray*}
	Taking the supremum over all $\lambda, \mu \in \Omega,$ we obtain the desired first inequality.
	\end{proof}
    
\begin{remark} 
For positive operators $A, B \in \mathcal{B}(\mathcal{H})$ and $t \in [0,1]$, we have the following inequality (see e.g. \cite{Axioms}):
\begin{eqnarray}\label{Axioms convex}
   \left\|tA + (1-t)B\right\|^r_{ber} \leq \left\| tA^r+(1-t)B^r\right\|_{ber}, \quad \forall r \geq 1. 
\end{eqnarray}
From the definition of the $t-$Berezin norm, it is observed that 
\begin{eqnarray}\label{new}
    \textbf{ber}(A) \leq \min_{t\in [0,1] } \|A\|_{t-ber}.
\end{eqnarray}
Now, combining the inequalities \eqref{Axioms convex},  \eqref{new} and Theorem \ref{th 3}, we get
\begin{eqnarray}\label{new-1}
    \textbf{ber}^r(A) \leq \min_{t\in [0,1] } \|A\|_{t-ber}^r \leq \frac12\left\| |A|^r + |A^*|^r \right\|_{ber}.
\end{eqnarray}
Clearly, $\min_{t\in [0,1] } \|A\|_{t-ber}^r$ is a stronger upper bound of $\textbf{ber}^r(A)$ than the existing bound 
    \begin{equation} \label{Taghavi}
        \textbf{ber}^r(A) \leq \frac12 \left\| |A|^r + |A^*|^r \right\|_{ber}, \quad \textit{for all $r\geq 1,$}
    \end{equation}
    proved by  Taghavi et al. \cite{Taghavi}. 
\end{remark}

Our next theorem shows that the bound \eqref{new} is stronger than $
    \textbf{ber}^2(A) \leq \frac12 \left\||A|^2 + |A^*|^2 \right\|_{ber} ,$
as studied in \cite[Theorem 3.2]{BMA}.

	\begin{theorem} \label{th 4}
		If $A \in \mathcal{B}(\mathcal{H})$, then \[\| A\|_{t-ber} \leq \sqrt{\| tA^*A+(1-t)AA^*\|_{ber}}.\]
        
        Moreover, if $t=\frac{1}{2},$ we have $\| A\|_{1/2-ber} \leq \sqrt{\frac{1}{2} \| A^*A+AA^*\|_{ber}} .$
        \end{theorem}

	\begin{proof}
    The first inequality will now be demonstrated, and the second inequality directly follows from it.
		 We have 
		\begin{eqnarray*}
			 &&\sup_{\lambda, \mu \in \Omega} \left(t|\langle A \hat{k}_\lambda, \hat{k}_\mu \rangle| + (1-t) |\langle \hat{k}_\lambda, A \hat{k}_\mu \rangle| \right)^2\\
			&\leq& \sup_{\lambda, \mu \in \Omega} \left(t|\langle A \hat{k}_\lambda, \hat{k}_\mu \rangle|^2 + (1-t) |\langle \hat{k}_\lambda, A \hat{k}_\mu \rangle|^2 \right) \quad (\mbox{by convexity of $f(t) = t^2,~ t \geq 0$} )\\
				&\leq&  \sup_{\lambda, \mu \in \Omega} \left(t\|A\hat{k}_\lambda\|^2 + (1-t) \|A^*\hat{k}_\lambda\|^2 \right) \quad (\mbox{by the Cauchy-Schwarz inequality})\\
					&=& \left\| tA^*A+(1-t)AA^*\right\|_{ber},
		\end{eqnarray*}
        as desired.
	\end{proof}

	\begin{remark}
    Following Theorem \ref{th 4} and \eqref{new}, we have $$\textbf{ber}(A) \leq \min_{t \in [0,1]} \|A\|_{t-ber} \leq \min_{t \in [0,1]}\sqrt{\| tA^*A+(1-t)AA^*\|_{ber}} \leq  \sqrt{\frac{1}{2} \| A^*A+AA^*\|_{ber}} .$$ To show strict refinement one can consider $A=\begin{pmatrix}
        1&1\\
        1&1
    \end{pmatrix}$ is an operator on the two-dimensional reproducing kernel Hilbert space $\mathbb{C}^2.$ Then $$\min_{t \in [0,1]} \|A\|_{t-ber} =1<\sqrt{2}=\min_{t \in [0,1]} \sqrt{\| tA^*A+(1-t)AA^*\|_{ber} }. $$
		\end{remark}

   Recently, Bhunia et al. \cite[Theorem 5]{CAOT} provided a characterization of the unitary operators in terms of the Berezin numbers, i.e. an invertible operator $A\in \mathcal{B}(\mathcal{H})$ is unitary if and only if $\textbf{ber}(A^*A) \leq 1$ and $\textbf{ber}((A^*A)^{-1}) \leq 1$.

   This naturally leads to the question: Can we characterize the unitary operators using our newly defined $t-$Berezin norm? The answer is yes! Indeed, we get the following proposition:   
   
   \begin{prop}
   Suppose $A \in \mathcal{B}(\mathcal{H})$ is invertible. Then $A $ is unitary if and only if $\|A^*A\|_{t-ber} \leq 1$ and $\|(A^*A)^{-1}\|_{t-ber} \leq 1.$
   \end{prop}

   \begin{proof}
   The necessity is trivial, as follows from the definition of the $t-$Berezin norm. 
   To show sufficiency, we calculate $\| (A-{A^{-1}}^*)\hat{k}_{\lambda}\|^2.$ Using 
the inequalities $\textbf{ber}(A^*A) \leq \|A^*A\|_{t-ber}\leq 1$ and $\textbf{ber}((A^*A)^{-1}) \leq \|(A^*A)^{-1}\|_{t-ber}\leq 1$, we obtain $\| (A-{A^{-1}}^*)\hat{k}_{\lambda}\|=0$ for all $\lambda\in \Omega.$ This gives $A^*=A^{-1}.$
   \end{proof}
	
\subsection*{Inequalities for operator matrices}
	Here we study the $t-$Berezin norm of $2 \times 2$ operator matrices $\left(\begin{array}{cc}
 A&B\\
 C&D
 \end{array}\right),$ where $A, B,C$ and $D \in \mathcal{B}(\mathcal{H}).$
 First we prove the following proposition:
	
	\begin{prop} \label{operator matrix t ber}
		Let $A, B \in \mathcal{B}(\mathcal{H}). $ Then the following results hold:
		\begin{enumerate}
			\item $\left\|\left(\begin{array}{cc}
				A&O\\
				O&B
			\end{array}\right)\right\|_{t-ber} \leq \max \left\{\|A\|_{t-ber}, \|B\|_{t-ber} \right\}$
			\item  $\left\| \begin{pmatrix}
				0&A\\
				B&0
			\end{pmatrix} \right \|_{t-ber}= \left\| \begin{pmatrix}
				0&B\\
				A&0
			\end{pmatrix} \right \|_{t-ber}$
			\item  $\left\| \begin{pmatrix}
				0&A\\
				0&0
			\end{pmatrix} \right \|_{t-ber}\leq \max \{t, 1-t\} \|A\|_{ber}.$
			\end{enumerate}
	\end{prop}

	\begin{proof}
	~~(1)~	Let $T=\begin{pmatrix}
			A&0\\
			0&B
		\end{pmatrix}.$ For any $(\lambda_1, \lambda_2),~ (\mu_1, \mu_2) \in \Omega \times \Omega,$ let $\hat{k}_{(\lambda_1, \lambda_2)} = (k_{\lambda_1}, k_{\lambda_2}),  $ $ \hat{\mu}_{(\mu_1, \mu_2)} = (k_{\mu_1}, k_{\mu_2})$ be two normalized reproducing kernels in $\mathcal{H} \oplus \mathcal{H}.$ Then 
		\begin{eqnarray*}
			&& t|\langle T\hat{k}_{(\lambda_1, \lambda_2)}, \hat{\mu}_{(\mu_1, \mu_2)}\rangle|+(1-t)|\langle T^*\hat{k}_{(\lambda_1, \lambda_2)}, \hat{\mu}_{(\mu_1, \mu_2)}\rangle|\\
			&=& t|\langle A k_{\lambda_1}, k_{\mu_1}\rangle+ \langle B k_{\lambda_2}, k_{\mu_2}\rangle|+(1-t) |\langle A^* k_{\lambda_1}, k_{\mu_1}\rangle+ \langle B^* k_{\lambda_2}, k_{\mu_2}\rangle|\\
			&\leq & t|\langle A k_{\lambda_1}, k_{\mu_1}\rangle|+  (1-t) |\langle A^* k_{\lambda_1}, k_{\mu_1}\rangle| + t|\langle B k_{\lambda_2}, k_{\mu_2}\rangle|  +(1-t) |\langle B^* k_{\lambda_2}, k_{\mu_2}\rangle|\\
			&\leq& \|A\|_{t-ber} \| k_{\lambda_1}\| \|k_{\mu_1}\|+ \|B\|_{t-ber} \|k_{\lambda_2}\| \| k_{\mu_2}\|\\
			&\leq& \max \{ \|A\|_{t-ber}, \|B\|_{t-ber}\} (\|k_{\lambda_1}\| \|k_{\mu_1}\|+ \|k_{\lambda_2}\| \| k_{\mu_2}\|)\\
			&\leq& \max \{ \|A\|_{t-ber}, \|B\|_{t-ber}\}.
		\end{eqnarray*}
		This gives the inequality (1).\\
		(2)~ The proof follows from the fact that $\left\|P^* \begin{pmatrix}
				0&A\\
				B&0
			\end{pmatrix} P \right  \|_{t-ber}= \left\| \begin{pmatrix}
				0&A\\
				B&0
			\end{pmatrix} \right \|_{t-ber},$ where $P=\begin{pmatrix}
			0 & I\\
			I & 0
		\end{pmatrix}\in \mathcal{B}(\mathcal{H}\oplus \mathcal{H}).$ \\	
		(3) Let $T=\begin{pmatrix}
			0&A\\
			0&0
		\end{pmatrix}.$ For any $(\lambda_1, \lambda_2), (\mu_1, \mu_2) \in \Omega \times \Omega,$ let $\hat{k}_{(\lambda_1, \lambda_2)} = (k_{\lambda_1}, k_{\lambda_2}), \hat{\mu}_{(\mu_1, \mu_2)} = (k_{\mu_1}, k_{\mu_2})$ be two  the normalized reproducing kernels in $\mathcal{H} \oplus \mathcal{H}.$  Then we have
	\begin{eqnarray*}
		&& t|\langle T\hat{k}_{(\lambda_1, \lambda_2)}, \hat{\mu}_{(\mu_1, \mu_2)}\rangle|+(1-t)|\langle T^*\hat{k}_{(\lambda_1, \lambda_2)}, \hat{\mu}_{(\mu_1, \mu_2)}\rangle|\\
		&=& t|\langle A k_{\lambda_2}, k_{\mu_1}\rangle+ (1-t) |\langle A^* k_{\lambda_1}, k_{\mu_2}\rangle\\
			&\leq & \|A\|_{ber} (t \|k_{\lambda_2}\| \|k_{\mu_1}\| + (1-t) \|k_{\lambda_1}\| \|k_{\mu_2}\| )\\
			&\leq& \max \{ t, 1-t \} \|A\|_{ber}  ( \|k_{\lambda_2}\| \| k_{\mu_1}\| +  \|k_{\lambda_1}\| \|k_{\mu_2}\| )\\
			&\leq& \max \{ t, 1-t \}  \|A\|_{ber}.
		\end{eqnarray*}
		This gives $\| T \|_{t-ber}\leq \max \{   t, 1-t\} \|A\|_{ber}.$
		\end{proof}
	
Building on this proposition, we establish an upper bound for $2 \times 2$ off-diagonal operator matrices.

\begin{cor}\label{op th1}
	Let $A,B \in \mathcal{B}(\mathcal{H})$. Then
\[ \left\| \begin{pmatrix}
		0 & A\\
		B& 0
	\end{pmatrix} \right\|_{t-ber} \leq \max\{  t, 1-t \} \left(\|A\|_{ber}+ \|B\|_{ber}\right).\]
\end{cor}	
	
\begin{proof}
		Using the norm property of the $t-$Berezin norm and Proposition \ref{operator matrix t ber}, we get
			\begin{eqnarray*}
			\left\|\begin{pmatrix}
			0 & A\\
			B& 0
		\end{pmatrix}\right \|_{t-ber}&\leq& \left\| \begin{pmatrix}
				0 & A\\
				0& 0
			\end{pmatrix} \right\|_{t-ber}+ \left\| \begin{pmatrix}
				0 & 0\\
				B& 0
			\end{pmatrix} \right\|_{t-ber}
			= \left\| \begin{pmatrix}
				0 & A\\
				0& 0
			\end{pmatrix} \right\|_{t-ber} +\left\| \begin{pmatrix}
				0 & B\\
				0& 0
			\end{pmatrix} \right\|_{t-ber} \\ 
			&\leq& \max \{ t,1-t\} (\|A\|_{ber}+\|B\|_{ber}). 
		\end{eqnarray*}
	\end{proof}

Next, we obtain an upper bound for general $2\times 2$ operator matrices.
	
	\begin{theorem}\label{op th4}
		Let $A, B, C, D \in \mathcal{B}(\mathcal{H}).$ Then
		$$ \left\| \begin{pmatrix}
			A & B\\
			C& D
		\end{pmatrix} \right\|_{t-ber} \leq \left\| \begin{pmatrix}
			\|A\|_{t-ber} & t\|B\|_{ber}+(1-t)\|C\|_{ber} \\
			t\|C\|_{ber}+(1-t)\|B\|_{ber}& \|D\|_{t-ber}
		\end{pmatrix} \right\| .$$
	\end{theorem}
    
	\begin{proof}
		Let $T= \begin{pmatrix}
			A & B\\
			C& D
		\end{pmatrix}.$ For any $(\lambda_1, \lambda_2),\, (\mu_1, \mu_2) \in \Omega \times \Omega,$ and let $\hat{k}_{(\lambda_1, \lambda_2)} = (k_{\lambda_1},  k_{\lambda_2}), $ $ \hat{\mu}_{(\mu_1, \mu_2)} = (k_{\mu_1}, k_{\mu_2})$ be two normalized reproducing kernels in $\mathcal{H} \oplus  \mathcal{H}.$  Then  we have
		\begin{eqnarray*}
			&&	t|\langle T\hat{k}_{(\lambda_1, \lambda_2)}, \hat{\mu}_{(\mu_1, \mu_2)}\rangle|+ (1-t) |\langle T^*\hat{k}_{(\lambda_1, \lambda_2)}, \hat{\mu}_{(\mu_1, \mu_2)}\rangle|\\
			&\leq& t|\langle A k_{\lambda_1}, k_{\mu_1} \rangle|+ (1-t) |\langle A^* k_{\lambda_1}, k_{\mu_1}\rangle|+ t|\langle D k_{\lambda_2}, k_{\mu_2}\rangle|+ (1-t) |\langle D^* k_{\lambda_2}, k_{\mu_2}\rangle|\\
			&& + (t\|B\|_{ber}+(1-t)\|C\|_{ber})\|k_{\lambda_2}\| \|k_{\mu_1}\|+ (t\|C\|_{ber}+(1-t)\|B\|_{ber})\| k_{\lambda_1}\| \|k_{\mu_2}\|\\
			&\leq& \|A\|_{t-ber} \|k_{\lambda_1}\|\|k_{\mu_1}\| + \|D\|_{t-ber} \|k_{\lambda_2}\|\|k_{\mu_2}\|\\
			&& + (t\|B\|_{ber}+(1-t)\|C\|_{ber})\|k_{\lambda_2}\| \|k_{\mu_1}\|+ (t\|C\|_{ber}+(1-t)\|B\|_{ber})\|k_{\lambda_1}\| \|k_{\mu_2}\|\\
			&=& \left \langle \begin{pmatrix}
				\|A\|_{t-ber} & t\|B\|_{ber}+(1-t)\|C\|_{ber} \\
				t\|C\|_{ber}+(1-t)\|B\|_{ber}& \|D\|_{t-ber}
			\end{pmatrix} \begin{pmatrix}
				\|k_{\lambda_1}\|\\
				\|k_{\lambda_2}\|
			\end{pmatrix}, \begin{pmatrix}
				\|k_{\mu_1}\|\\
				\|k_{\mu_2}\|
			\end{pmatrix} \right \rangle \\
			&\leq& \left\| \begin{pmatrix}
				\|A\|_{t-ber} & t\|B\|_{ber}+(1-t)\|C\|_{ber} \\
				t\|C\|_{ber}+(1-t)\|B\|_{ber}& \|D\|_{t-ber}
			\end{pmatrix} \right\|.
		\end{eqnarray*}
		Therefore, taking the supremum over all $\lambda, \mu \in \Omega,$ we obtain the desired bound.
	\end{proof}
	
	Analogous to  Theorem \ref{op th4}, we can derive an upper bound for general $n\times n$ operator matrices. 
	
	\begin{cor}
		Let $A=(A_{ij})_{i,j=1}^n$ be an $n\times n$ operator matrix, with the entries $A_{ij}\in \mathcal{B}(\mathcal{H})$. Then 
		$ \|T\|_{t-ber} \leq \| {T}'\|,$
		where ${T}'= (t_{ij})$ is an $n\times n$ matrix has the entries 
		$$t_{ij}=
		\begin{cases}
			\|T_{ij}\|_{t-ber}  &  \text{if $i=j$},\\
			t\|T_{ij}\|_{ber}+ (1-t)\|T_{ji}\|_{ber}  & \text{if $i\neq j$}.
		\end{cases} $$
	\end{cor}

			\section{Berezin number inequalities via Orlicz functions}
            
          In this section, we present upper bounds for the Berezin number of bounded linear operators utilizing Orlicz functions. Recall that, an Orlicz function \(\phi:[0, \infty) \rightarrow [0, \infty )\) is a function that satisfies the following conditions: 
          (i) $\phi$ is continuous, convex and increasing; 
          (ii) $\phi(0) = 0$; 
          (iii) $\phi(t) \rightarrow +\infty$ as $t \rightarrow +\infty.$ Moreover, an Orlicz function is said to be sub-multiplicative if for all $x, y \geq 0,$ $\phi(xy) \leq \phi(x) \phi (y)$ holds.

         In the first theorem of this section, we will improve (and generalize) the inequality \eqref{Taghavi}.
	
		\begin{theorem} \label{th 6}
		Let $A \in \mathcal{B}(\mathcal{H}).$ Let $g, h$ be two nonnegative continuous functions on $[0, \infty) $ satisfying $g(t) h(t)=t$ for all $t\geq 0$. Then,  for a well-defined function $f:(0,1) \rightarrow [0,\infty)$ and a sub-multiplicative Orlicz function $\phi$,
		\begin{eqnarray*}
			\phi\left(	\textbf{ber}^2(A)  \right)  &\leq& \frac{f(t)}{2(1+ f(t))}  \left(\frac12  \textbf{ber} \left(\phi\left(g^4(|A|)\right) + \phi\left(h^4(|A^*|)\right)\right) +  \phi\left(\textbf{ber}\left( h^2(|A^*|) g^2(|A|)\right)\right) \right) \\&&+ \frac{1}{2(1+ f(t))} \phi\left(\textbf{ber}(A)\right) ber\left(\phi \left(g^2(|A|)\right) + \phi\left(h^2(|A^*|)\right) \right).
		\end{eqnarray*}
	\end{theorem}

    Before proceeding with the proof of this theorem, we derive the following corollaries.

	\begin{cor} \label{th 6 cor1}
    Let $A \in \mathcal{B}(\mathcal{H})$ and $r \geq 1,$ then
		\begin{eqnarray*}
			\textbf{ber}^{2r}(A) &\leq& \frac18 \left\||A|^{2r} + |A^*|^{2r} \right\|_{ber} + \frac14 \textbf{ber}^r\left(|A^*||A|\right) + \frac14 \textbf{ber}^r(A) \left\||A|^r + |A^*|^r \right\|_{ber}.
		\end{eqnarray*}
        \end{cor}
\begin{proof}
This follows from Theorem \ref{th 6} by taking $\phi(t) =t^r$, $g(t) = h(t)= \sqrt{t}$ and $f(t)=1.$ 
\end{proof}

	\begin{cor}\label{th 6 cor2}
		Let $A \in \mathcal{B}(\mathcal{H}),$ then \begin{eqnarray*}
			\textbf{ber}^2(A) &\leq& \frac{1}{12} \left\||A|^2 + |A^*|^2\right\|_{ber} + \frac{1}{6} \textbf{ber}\left(|A^*||A|\right) + \frac{1}{3} \textbf{ber}(A) \left\||A| + |A^*|\right\|_{ber}.
		\end{eqnarray*}
	\end{cor}

    \begin{proof}
        This follows from Theorem \ref{th 6} by setting $\phi(t)= t$, $g(t) = h(t) = \sqrt{t}$ and $ f(t) = \frac{1}{2}.$ 
    \end{proof}

\begin{remark}
(1) Using the inequality \eqref{Basaran} (below) and the equality \eqref{Pintu CAOT}, we conclude that Corollary \ref{th 6 cor1} improves the bound \eqref{Taghavi} for $r \geq 2.$

\noindent (2) The inequality in Corollary \ref{th 6 cor2} significantly improves \cite[Theorem 3]{Axioms}, namely,  \[\textbf{ber}^2(A) \leq \frac16 \left\||A|^2 + |A^*|^2 \right\|_{ber} + \frac13 \textbf{ber}(A) \left\||A| + |A^*|\right\|_{ber}.\] 
\end{remark}

We now prove Theorem \ref{th 6}. To do this, first, we need the following lemma, which is known as Buzano's inequality (an extension of the Cauchy-Schwarz inequality).
    
\begin{lemma}\label{buzano}\cite{Buzano}
		Let $x, y, e \in \mathcal{H}$ with $\|e\|=1.$ Then  $|\langle x, e\rangle \langle e , y\rangle| \leq \frac{1}{2}\left(\|x\|\|y\| + |\langle x, y \rangle| \right) .$
	\end{lemma}

\begin{proof}[Proof of Theorem \ref{th 6}]
		Let $\hat{k}_\lambda$ be a normalized reproducing kernel of $\mathcal{H}.$	Employing the convexity property of $\phi,$ we obtain
		\begin{eqnarray*}
			&& \phi \left(\left|\langle A\hat{k}_\lambda, \hat{k}_\lambda \rangle\right|^2  \right)\\ 
			&\leq& \frac{f(t)}{1+f(t)} \phi \left( \left|\langle A\hat{k}_\lambda, \hat{k}_\lambda \rangle\right|^2 \right) + \frac{1}{1+f(t)}\phi \left( \left|\langle A\hat{k}_\lambda, \hat{k}_\lambda \rangle\right|^2 \right)\\
			&\leq& \frac{f(t)}{1+ f(t)} \phi \left( \left\langle g^2(|A|)\hat{k}_\lambda, \hat{k}_\lambda \right \rangle \left \langle h^2(|A^*|) \hat{k}_\lambda, \hat{k}_\lambda \right \rangle\right) \\ &&+ \frac{1}{1+ f(t)} \phi \left(|\langle A\hat{k}_\lambda, \hat{k}_\lambda \rangle| \sqrt{\left\langle g^2(|A|)\hat{k}_\lambda, \hat{k}_\lambda \right \rangle \left \langle h^2(|A^*|) \hat{k}_\lambda, \hat{k}_\lambda  \right \rangle} \right) \, (\mbox{using Lemma \ref{mixed schwarz}})\\
			&\leq& \frac{f(t)}{1+ f(t)} \phi \left(\frac{\left\|g^2(|A|)\hat{k}_\lambda \right\| \left\| h^2(|A^*|)\hat{k}_\lambda \right\| + \left|\left\langle g^2(|A|)\hat{k}_\lambda, h^2(|A^*|)\hat{k}_\lambda \right\rangle\right|}{2} \right) \\ && + \frac{1}{1+ f(t)} \phi \left(|\langle A\hat{k}_\lambda, \hat{k}_\lambda \rangle| \right) \phi \left(\frac{ \left\langle g^2(|A|)\hat{k}_\lambda, \hat{k}_\lambda \right \rangle + \left \langle h^2(|A^*|) \hat{k}_\lambda, \hat{k}_\lambda \right \rangle}{2} \right)\\
			&&\,\,\,\,\,\,\,\,\,\,\,\,\,\,\,\,\,\,\,\,\,\,\,\,\,~~(\mbox{using Lemma \ref{buzano} and sub-multiplicative property of $\phi$})\\
			&\leq& \frac{f(t)}{2(1+ f(t))} \phi \left(\left\|g^2(|A|)\hat{k}_\lambda \right\| \left\| h^2(|A^*|)\hat{k}_\lambda \right\| \right) + \frac{f(t)}{2(1+ f(t))} \phi \left(\left|\left \langle h^2(|A^*|) g^2 (|A|) \hat{k}_\lambda, \hat{k}_\lambda \right \rangle\right| \right)\\ && + \frac{1}{2(1+ f(t))} \phi \left(\left|\langle A\hat{k}_\lambda, \hat{k}_\lambda \rangle\right| \right) \left \langle \left(\phi \left(g^2(|A|)\right) + \phi \left(h^2(|A^*|)\right)\right)\hat{k}_\lambda, \hat{k}_\lambda \right\rangle \\
			&\leq& \frac{f(t)}{2(1+ f(t))} \left(\phi \left(\left\langle \left( \frac{g^4(|A|) + h^4(|A^*|)}{2} \right)\hat{k}_\lambda, \hat{k}_\lambda \right\rangle \right) +  \phi \left(\left|\left \langle h^2(|A^*|) g^2 (|A|) \hat{k}_\lambda, \hat{k}_\lambda \right \rangle\right| \right)\right)\\ &&  + \frac{1}{2(1+ f(t))} \phi \left(\left|\langle A\hat{k}_\lambda, \hat{k}_\lambda \rangle\right| \right) \left \langle \left( \phi \left(g^2(|A|)\right) + \phi \left(h^2(|A^*|)\right)\right)\hat{k}_\lambda, \hat{k}_\lambda \right\rangle \\
			&\leq& \frac{f(t)}{4(1+ f(t))} \left(\left\langle \left(\phi\left(g^4(|A|)\right) + \phi\left(h^4(|A^*|) \right)\right) \hat{k}_\lambda, \hat{k}_\lambda  \right\rangle  + 2 \phi \left(\left \langle h^2(|A^*|) g^2 (|A|) \hat{k}_\lambda, \hat{k}_\lambda \right \rangle \right)\right)\\ &&  + \frac{1}{2(1+ f(t))} \phi \left(|\langle A\hat{k}_\lambda, \hat{k}_\lambda \rangle| \right) \left \langle \left( \phi \left(g^2(|A|)\right) + \phi \left(h^2(|A^*|)\right)\right)\hat{k}_\lambda, \hat{k}_\lambda \right\rangle \\
			&\leq& \frac{f(t)}{4(1+ f(t))}\textbf{ber} \left(\phi\left(g^4(|A|)\right) + \phi\left(h^4(|A^*|)\right)\right) + \frac{f(t)}{2(1+ f(t))} \phi\left(\textbf{ber}\left( h^2(|A^*|) g^2(|A|)\right)\right) \\&&+ \frac{1}{2(1+ f(t))} \phi\left(\textbf{ber}(A)\right) \textbf{ber}\left(\phi \left(g^2(|A|)\right) + \phi\left(h^2(|A^*|)\right) \right).
		\end{eqnarray*}
        Therefore, we get the desired inequality by taking the supremum over all $\lambda \in \Omega$. 
	\end{proof}

Our next result improves (and generalizes) the following inequality (see in \cite{BMA}):
    \begin{equation} \label{Basaran}
       \textbf{ber}^r(B^*A) \leq \frac12 \left\| |A|^{2r} + |B|^{2r} \right\|_{ber}, \quad r\geq 1.
\end{equation}

\begin{theorem} \label{th 7}
		Let $A, B \in \mathcal{B}(\mathcal{H}).$ Let $f: (0,1)\rightarrow [0,\infty)$ be a well-defined function. Then for any sub-multiplicative Orlicz function $\phi$, 
		\begin{eqnarray*}
			\phi\left(\textbf{ber}^2(A^*B)\right) &\leq& \frac{1}{2(1+ f(t))} \phi\left(\textbf{ber}(A^*B)\right) \textbf{ber}\left( \phi(|A|^2) + \phi(|B|^2) \right) \\ && + \frac{f(t)}{2(1+f(t))} \left(  \phi\left(\textbf{ber}(|B|^2 |A|^2)\right) + \frac{1}{2} \textbf{ber} \left(\phi(|A|^4) + \phi(|B|^4) \right) \right).
		\end{eqnarray*}
	\end{theorem}

    Before we prove this, we deduce the following corollaries.
	
	\begin{cor} \label{th 7 cor 1}
		Let $A, B \in \mathcal{B}(\mathcal{H})$ and $\alpha \geq 0.$ Then for $r \geq 1,$ 
        \begin{eqnarray*}
			\textbf{ber}^{2r}(A^*B) &\leq& \frac{1}{2(1+ \alpha)} \textbf{ber}^r (A^*B)  \left\||A|^{2r} + |B|^{2r}\right\|_{ber} +  \frac{\alpha}{4(1+ \alpha)} \left\||A|^{4r} + |B|^{4r}\right\|_{ber} \\ && + \frac{\alpha}{2(1+ \alpha)}  \textbf{ber}^r\left( |B|^{2} |A|^{2}\right)\\
            &\leq& \frac{1}{2( 1+\alpha )} \textbf{ber}^r (A^*B)  \left\||A|^{2r} + |B|^{2r}\right\|_{ber} + \frac{\alpha}{2(1+\alpha)} \left\||A|^{4r} + |B|^{4r}\right\|_{ber}.
		\end{eqnarray*}	     
	\end{cor}
\begin{proof}
    Considering the Orlicz function $\phi(t) = t^r, t \geq 0,$ and $f(t) = \alpha,$ in Theorem \ref{th 7}, we get the first inequality. The second inequality follows via the inequality \eqref{Basaran}.
\end{proof}

	\begin{cor} \label{th 7 cor 2}
		Let $A, B \in \mathcal{B}(\mathcal{H}).$ Then
		\begin{eqnarray*}
			\textbf{ber}^{2}(A^*B) 
			&\leq&  \frac{1}{3} \left\||A|^{2} + |B|^{2}\right\|_{ber}\textbf{ber}(A^*B) + \frac{1}{12} \left\||A|^{4} + |B|^{4}\right\|_{ber} + \frac16 \textbf{ber}(|B|^2|A|^2).\\
            &\leq& \frac16 \left\| |A|^4 + |B|^4 \right\|_{ber} + \frac13 \textbf{ber} (A^*B) \left\| |A|^2 + |B|^2 \right\|_{ber}.
		\end{eqnarray*}
     \end{cor}
\begin{proof}
     The first inequality is derived from Theorem \ref{th 7} by selecting the Orlicz function $\phi(t) = t$ for $t \geq 0$ and setting $f(t) = \frac{1}{2}$. For the second inequality, we utilize \eqref{Pintu CAOT} and \eqref{Basaran}.
\end{proof}

   \begin{remark}  
   (1) By setting the Orlicz function $\phi(t) = t^r, r\geq 1,$ and $f(t) = \frac{t}{1-t}$ in Theorem \ref{th 7}, we derive the following inequality: for $A, B \in \mathcal{B}(\mathcal{H}),$ $0 \leq \alpha\leq 1$ and $r \geq 1,$ 
\begin{eqnarray*}
\textbf{ber}^{2r}(A^*B) &\leq& \frac{1-\alpha}{2}\textbf{ber}^r (A^*B)  \left\||A|^{2r} + |B|^{2r}\right\|_{ber} +  \frac{\alpha}{2} \left\||A|^{4r} + |B|^{4r}\right\|_{ber},
\end{eqnarray*}
which was also studied in  \cite[Theorem 3.1]{MJM}.

\noindent (2) Using \eqref{Basaran} and \eqref{Axioms convex},  we can show that the inequality in Corollary \ref{th 7 cor 1} provides a sharper bound compared to that of Gao et al. in \cite[Theorem 3.9]{DCDS}, namely,  \[\textbf{ber}^r(A^*B) \leq \frac{1}{2 \lambda + 2} \left\||A|^r + |B|^r \right\|_{ber} \textbf{ber}^{\frac{r}{2}}(A^*B) + \frac{\lambda}{2 \lambda +2} \left\||A|^{2r} + |B|^{2r} \right\|_{ber}, \] for any $\lambda \geq 0$ and $r \geq 2.$ Also, Corollary \ref{th 7 cor 1} improves inequality \eqref{Basaran}.

    \noindent (3) Also, Corollary \ref{th 7 cor 2} extends and refines the result established by Altwaijry et al. in \cite[Theorem 4]{Axioms}, namely, \[\textbf{ber}^2(A^*B) \leq \frac16 \left\||A|^4 + |B|^4\right\|_{ber} + \frac13 \textbf{ber}(A^*B) \left\| |A|^2 + |B|^2 \right\|_{ber}. \]
       \end{remark}

We now prove Theorem \ref{th 7}. For this we need the following lemma, which is a refinement of the Cauchy-Schwarz inequality.
	
	\begin{lemma}\cite[Lemma 2.7]{Nayak C-S}\label{gen cauchy}
		Let $f: (0,1) \rightarrow [0,\infty)$ be a well-defined function. Then,  \[|\langle x, y \rangle|^2 \leq \frac{f(t)}{1+f(t)} \|x\|^2 \|y\|^2  + \frac{1}{1+ f(t)} |\langle x, y \rangle | \|x\|\|y\| \quad \text{for any $x, y \in \mathcal{H}$}.\] 
	\end{lemma}

	\begin{proof}[Proof of  Theorem \ref{th 7}]
		Let $\hat{k}_\lambda$ be a normalized reproducing kernel of $ \mathcal{H}.$  By the convexity of $\phi$, we get
        \begin{eqnarray*}
			&& \phi\left(	\left|\langle A^*B\hat{k}_\lambda, \hat{k}_\lambda \rangle \right|^2 \right)\\ 
			&\leq& \phi \left( \frac{1}{(1+ f(t))}  \|A\hat{k}_\lambda \|\|B\hat{k}_\lambda\| |\langle A\hat{k}_\lambda, B\hat{k}_\lambda \rangle| +  \frac{f(t)}{(1+ f(t))} \|A\hat{k}_\lambda\|^2\|B\hat{k}_\lambda\|^2 \right)\, (\mbox{by Lemma \ref{gen cauchy}})\\
			&\leq& \frac{1}{(1+ f(t))} \phi\left( \|A\hat{k}_\lambda\| \|B\hat{k}_\lambda\||\langle A^*B\hat{k}_\lambda, \hat{k}_\lambda \rangle|\right) + \frac{f(t)}{(1+ f(t))} \phi\left(\|A\hat{k}_\lambda\|^{2} \|B\hat{k}_\lambda\|^{2} \right)\\
			&\leq& \frac{1}{(1+ f(t))} \phi \left(\left\langle \left(\frac{|A|^{2}+ |B|^{2}}{2} \right)\hat{k}_\lambda, \hat{k}_\lambda \right\rangle \right) \phi \left(\left|\langle A^*B\hat{k}_\lambda, \hat{k}_\lambda \rangle \right| \right)\\&& + \frac{f(t)}{(1+ f(t))} \phi\left(\langle |A|^{2}\hat{k}_\lambda, \hat{k}_\lambda \rangle \langle \hat{k}_\lambda, |B|^{2}\hat{k}_\lambda \rangle \right) \\&&\,\,\,\,\,\,\,\,\,\,\,\,\,\,\,\,\,\, \,\,\,\,\,\, \,\,\,\,\,\,\,\,\,\,\,\,\,\,\,\,\,\,\,\,\,\,\,\, \,\,\,\,\,\,\,\,\,\,\,\,\,\,\,\,\,\, \,\,\,\,\,\, \,\,\,\,\,\,\,\,\,\,\,\,\,\,\,\,\,\,(\mbox{by AM-GM inequality and sub-multiplicative of $\phi$})\\ 
			&\leq& \frac{1}{2(1+ f(t))}  \left\langle \left(\phi\left(|A|^{2}\right)+ \phi\left(|B|^{2}\right) \right)\hat{k}_\lambda, \hat{k}_\lambda \right\rangle \phi\left( \left|\langle A^*B\hat{k}_\lambda, \hat{k}_\lambda \rangle \right|\right) \\&&+ \frac{f(t)}{(1+ f(t))} \phi \left( \frac{\left\||A|^{2}\hat{k}_\lambda \right\|\left\||B|^{2}\hat{k}_\lambda \right\| + \left|\left\langle |A|^{2}\hat{k}_\lambda, |B|^{2}\hat{k}_\lambda \right\rangle\right| }{2} \right) \, (\mbox{using Lemma \ref{buzano}})\\
			&\leq&\frac{1}{2(1+ f(t))}  \left\langle \left(\phi\left(|A|^{2}\right)+ \phi\left(|B|^{2}\right) \right)\hat{k}_\lambda, \hat{k}_\lambda \right\rangle \phi\left( \left|\langle A^*B\hat{k}_\lambda, \hat{k}_\lambda \rangle \right|\right)  \\&&+ \frac{f(t)}{2(1+ f(t))}\left( \phi \left( \left\langle \left(\frac{|A|^{4} + |B|^{4} }{2} \right)\hat{k}_\lambda, \hat{k}_\lambda \right\rangle \right)+  \phi \left(\left|\left\langle \left(|B|^{2} |A|^{2}\right)\hat{k}_\lambda, \hat{k}_\lambda \right\rangle\right|\right)\right)\\
			&\leq&	\frac{1}{2(1+ f(t))}  \left\langle \left(\phi\left(|A|^{2}\right)+ \phi\left(|B|^{2}\right) \right)\hat{k}_\lambda, \hat{k}_\lambda \right\rangle \phi\left( \left|\langle A^*B\hat{k}_\lambda, \hat{k}_\lambda \rangle \right|\right)  \\&&+ \frac{f(t)}{4(1+ f(t))}  \left\langle \left(\phi\left(|A|^{4}\right) + \phi \left(|B|^{4}\right)  \right)\hat{k}_\lambda, \hat{k}_\lambda \right\rangle + \frac{f(t)}{2(1+f(t))} \phi \left(\left|\left\langle \left(|B|^{2} |A|^{2}\right)\hat{k}_\lambda, \hat{k}_\lambda \right\rangle\right|\right)\\
			&\leq& \frac{1}{2(1+ f(t))} \phi\left(\textbf{ber}(A^*B)\right) \textbf{ber}\left( \phi(|A|^2) + \phi(|B|^2) \right) \\ && + \frac{f(t)}{2(1+f(t))} \phi\left(\textbf{ber}(|B|^2 |A|^2)\right) + \frac{f(t)}{4(1+f(t))} \textbf{ber} \left(\phi(|A|^4) + \phi(|B|^4) \right).
		\end{eqnarray*}
		Therefore, the desired inequality follows by taking the supremum over all $\lambda \in \Omega$.
	\end{proof}

	Our next theorem is a generalization of $\textbf{ber}^{2}(A) \leq \frac12 \left\| |A|^{2} + |A^*|^{2} \right\|_{ber}$ (proved by Basaran et al. in \cite[Th. 3.2]{BMA})  through the Orlicz function.
	
		\begin{theorem}\label{th 8}
		Let $A \in \mathcal{B}(\mathcal{H}).$ Let $g,h$ be non-negative continuous functions on $[0, \infty)$ satisfying $g(t) h(t)=t,$ $\forall$ $t\geq 0.$ Then for any Orlicz function $\phi$ and for any $\alpha \in [0,1]$,
		\[\phi\left(\textbf{ber}^2(A) \right) \leq \textbf{ ber}\left(\frac{\alpha}{2}\left( \phi \left(g^4(|A|)\right) + \phi \left(h^4(|A^*|) \right)\right) + (1-\alpha) \phi \left(|A|^2\right)\right) \]
		and 
		\[\phi\left(\textbf{ber}^2(A) \right) \leq  \textbf{ber}\left(\frac{\alpha}{2}\left( \phi \left(g^4(|A^*|)\right) + \phi \left(h^4(|A|) \right)\right) + (1-\alpha) \phi \left(|A^*|^2\right)\right) .\]
	\end{theorem}
    
	\begin{proof}
	Let $\hat{k}_\lambda$ be a normalized reproducing kernel of $\mathcal{H}.$	Then by  the convexity of $\phi,$ we have 
		\begin{eqnarray*}
			\phi \left(\left|\left\langle A\hat{k}_\lambda, \hat{k}_\lambda \right\rangle\right|^2 \right) 
            &\leq& \alpha \phi \left(\left|\left\langle A\hat{k}_\lambda, \hat{k}_\lambda \right\rangle\right|^2\right) + (1-\alpha) \phi \left(\|A^*\hat{k}_\lambda\|^2\right)\\
			&\leq& \alpha \phi \left(\left\langle g^2(|A|)\hat{k}_\lambda,\hat{k}_\lambda \right\rangle \left\langle h^2(|A^*|)\hat{k}_\lambda, \hat{k}_\lambda \right\rangle  \right) + (1-\alpha) \left\langle |A^*|^2 \hat{k}_\lambda, \hat{k}_\lambda \right\rangle\\ 
&&\,\,\,\,\,\,\,\,\,\,\,\,\,\,\,\,\,\,\,\,\,\,\,\,\,\,\,\,\,\,\,\,\,\,\,\,\,\,\,\,\,\,\,\,\,\,\,\,\,\,\,\,\,\,\,\,\,\,\,\,\,\,\,\,\,\,\,\,\,\,~~(\mbox{using Lemma \ref{mixed schwarz}})\\
			&\leq& \alpha \phi \left(\left\langle\left(\frac{g^4(|A|) + h^4(|A^*|)}{2}\right)\hat{k}_\lambda, \hat{k}_\lambda \right\rangle \right) + (1-\alpha) \phi \left(\left\langle |A^*|^2\hat{k}_\lambda, \hat{k}_\lambda \right\rangle\right)\\
			&\leq& \frac{\alpha}{2} \left\langle \left(\phi\left(g^4(|A|)\right) + \phi\left(h^4(|A^*|)\right)\right)\hat{k}_\lambda, \hat{k}_\lambda  \right\rangle + (1-\alpha)\left \langle \phi\left(|A^*|^2\right) \hat{k}_\lambda, \hat{k}_\lambda \right\rangle\\ 
			&=& \left\langle \left(\frac{\alpha}{2} \left(\phi\left(g^4(|A|)\right) + \phi\left(h^4(|A^*|)\right)\right) + (1-\alpha) \phi \left(|A^*|^2\right) \right)\hat{k}_\lambda, \hat{k}_\lambda \right\rangle
			\\ &\leq&  \textbf{ber}\left(\frac{\alpha}{2}\left(\phi \left(g^4(|A|)\right) + \phi \left(h^4(|A^*|) \right)\right) + (1-\alpha) \phi \left(|A^*|^2\right)\right).
		\end{eqnarray*}
		Therefore, taking the supremum over all $\lambda \in \Omega$, we obtain the first inequality. The second inequality can be obtained by replacing $A$ with $A^*$.	
	\end{proof}

\smallskip

	 \textbf{Data availability statements.} No data was used for the research described in the article.
     
	\textbf{Declaration of competing interest.} There is no competing interest.
    

	\bibliographystyle{amsplain}
	
\end{document}